\documentclass[11pt,a4paper]{article}
\usepackage[T1]{fontenc}
\usepackage[utf8]{inputenc}
\usepackage{lmodern}
\usepackage[margin=30mm]{geometry}
\usepackage{amsmath,amssymb,amsthm}
\usepackage{enumitem}
\usepackage{microtype}
\usepackage{hyperref}
\hypersetup{hidelinks,
  pdftitle={A lower bound on the density of prefixes with maximal palindromic length},
  pdfauthor={Josef Rukavicka},
  pdfsubject={Combinatorics on words; palindromic length},
  pdfkeywords={palindromic length, palindromic prefix, palindromic factorization}}
\allowdisplaybreaks[1]
\numberwithin{equation}{section}

\theoremstyle{plain}
\newtheorem{theorem}{Theorem}[section]
\newtheorem{proposition}[theorem]{Proposition}
\newtheorem{lemma}[theorem]{Lemma}

\theoremstyle{definition}
\newtheorem{definition}[theorem]{Definition}
\newtheorem{example}[theorem]{Example}
\theoremstyle{remark}
\newtheorem{remark}[theorem]{Remark}

\DeclareMathOperator{\Suffix}{Suf}
\DeclareMathOperator{\Prefix}{Prf}
\DeclareMathOperator{\Factor}{Fac}
\DeclareMathOperator{\Pal}{Pal}
\DeclareMathOperator{\PL}{PL}
\DeclareMathOperator{\maxPrefixPL}{maxPrfPL}
\DeclareMathOperator{\occur}{occur}
\DeclareMathOperator{\mper}{mper}
\DeclareMathOperator{\order}{order}
\DeclareMathOperator{\Mirror}{Mirror}
\DeclareMathOperator{\red}{red}
\DeclareMathOperator{\PalFac}{PalFac}
\DeclareMathOperator{\cut}{cut}
\newcommand{\FF}{F}
\newcommand{\doi}[1]{\href{https://doi.org/#1}{doi:\nolinkurl{#1}}}

\title{A lower bound on the density of prefixes\protect\\
with maximal palindromic length}
\author{Josef Rukavicka\thanks{Department of Mathematics,
Faculty of Nuclear Sciences and Physical Engineering, Czech Technical University in Prague
(josef.rukavicka@seznam.cz).}}

\date{\small{September 14, 2026}}

\begin{document}
\maketitle

\begin{abstract}
The palindromic length $\PL(u)$ of a nonempty finite word $u$ is the
least number of nonempty palindromes whose concatenation is $u$.
For an infinite word $w$, let $P_w(n)$ be the number of nonempty
palindromic prefixes of $w[1,n]$, and let $T_w(n)$ consist of those
prefixes $w[1,j]$, $1\leq j\leq n$, whose palindromic length equals the
maximum attained among the nonempty prefixes of $w[1,j]$.
We prove the finite inequality $|T_w(n)|\geq P_w(n)$ for every $n\geq1$.
In particular, if $w$ has infinitely many palindromic prefixes, then
\[
 \liminf_{n\to\infty}\frac{|T_w(n)|}{P_w(n)}\geq1.
\]
The coefficient $1$ is optimal, already for a nonconstant periodic
word.  The proof uses chains of occurrences connected by palindromic
factors, together with trimming and reflection arguments that preserve
lower bounds on palindromic length.
\end{abstract}



\section{Introduction}
\label{sec:introduction}

The \emph{palindromic length} of a nonempty finite word is the least
number of nonempty palindromes whose concatenation is that word.
This parameter connects the local occurrence of palindromes with the
factorization structure of finite and infinite words.  Frid, Puzynina,
and Zamboni~\cite{FPZ2013} initiated a systematic study of bounded
palindromic length in infinite words.  They conjectured that an infinite
word whose factors have uniformly bounded palindromic length is
ultimately periodic, and established, in particular, the obstruction
provided by power-free words.  Saarela~\cite{Saarela2017} proved that
boundedness for prefixes is equivalent to boundedness for all factors:
if the palindromic length of every prefix is at most $K$, then that of
every factor is at most $2K$.  The conjecture was subsequently proved by
Rukavicka~\cite{Rukavicka2026}.  These results concern whether
palindromic length is bounded; they do not by themselves quantify how
often a prefix attains the largest value seen so far.

More detailed information is available for several structured families
of infinite words.  Frid~\cite{Frid2018} used numeration systems associated
with Sturmian words to prove unboundedness of their prefix palindromic
length.  For the Thue--Morse word, Frid~\cite{Frid2019} determined the
prefix palindromic-length sequence explicitly and proved that it is
$2$-regular.  Frid, Laborde, and Peltom\"aki~\cite{FLP2021} studied
prefix palindromic length of automatic words, proving regularity of this sequence when the word has only finitely
many distinct palindromic factors, and obtaining
explicit descriptions for the paperfolding and Rudin--Shapiro words.
A different approach is provided by the left and right greedy
factorizations studied by Bucci and Richomme~\cite{BR2018}, who proved
ultimate periodicity under boundedness of the corresponding greedy
prefix lengths.  On the algorithmic side, Fici, Gagie, K\"arkk\"ainen,
and Kempa~\cite{FGKK2014} gave an $O(n\log n)$-time algorithm for minimum
palindromic factorization.  These contributions illustrate the
structural and quantitative questions surrounding the sequence of
palindromic lengths of prefixes.

The lengths of palindromic prefixes form another natural source of
quantitative information.  Fischler~\cite{Fischler2006} studied infinite
words with abundant palindromic prefixes, described the words whose
successive palindromic-prefix lengths satisfy
$a_{i+1}\leq 2a_i+1$, and investigated ratios of successive lengths.
Here we impose no such growth condition.  Instead, we compare the
number of palindromic prefixes with the number of prefixes attaining
the running maximum of palindromic length.

Let $w$ be a right-infinite word over a finite alphabet.  For a nonempty
finite word $v$, write
\[
 \maxPrefixPL(v)=\max\{\PL(s):s\in\Prefix(v)\cap\Sigma^+\}.
\]
For $n\geq1$, let $P_w(n)$ be the number of nonempty palindromic prefixes
of $w[1,n]$, and put
\[
 T_w(n)=\{w[1,j]:1\leq j\leq n,
             \PL(w[1,j])=\maxPrefixPL(w[1,j])\}.
\]
A prefix belongs to $T_w(n)$ when it attains the running maximum at its
own length; ties with an earlier maximum are included.  Thus this is
not the set of strict record positions, nor the set of prefixes
attaining the maximum over the entire interval $[1,n]$.

Our main result, Proposition~\ref{nfh76dg347g}, is the finite inequality
\begin{equation}
\label{eq:introduction-main}
 |T_w(n)|\geq P_w(n)\qquad(n\geq1),
\end{equation}
valid for every right-infinite word $w$.  In particular, when $w$ has
infinitely many palindromic prefixes,
\[
 \liminf_{n\to\infty}\frac{|T_w(n)|}{P_w(n)}\geq1.
\]
The coefficient $1$ is best possible, already for the periodic word
$(ab)^\infty$ with $a\neq b$; see Example~\ref{ex:sharpness}.
No aperiodicity, recurrence, or power-avoidance assumption is needed.
The normalization is by $P_w(n)$ rather than by $n$: the result is a
relative counting bound and does not assert positive natural density
of the selected positions.  Nor do we assert that the quotient has an
ordinary limit.

The proof proceeds through a class $\Gamma(t)$ defined by chains of
occurrences of $t$ and $t^R$ connected by palindromic factors.  We prove
that
\[
 u\in\Gamma(t)\quad\Longrightarrow\quad\PL(u)\geq\PL(t)
\]
and establish trimming properties for these chains.  As a consequence,
if $w[1,a]$ and $w[1,b]$ are palindromes with $a<b$, then
\[
 \PL(w[1,r+b-a])\geq\PL(w[1,r])\qquad(1\leq r\leq a).
\]
This translation property places an index attaining the running maximum
between every two consecutive palindromic-prefix lengths.  Choosing
such indices in disjoint intervals proves
\eqref{eq:introduction-main}.

The paper is organized as follows.  Section~\ref{sec:preliminaries}
fixes notation.  Section~\ref{sec:palindromic-concatenation} introduces
$\Gamma(t)$ and proves an occurrence-count identity.
Section~\ref{sec:tailing} studies the reflection iteration and the
occurrences that cross palindromic-factorization boundaries.
Section~\ref{sec:trimming} establishes trimming and the
palindromic-length comparison.  Section~\ref{sec:density} proves the
counting bound and its sharpness.

\section{Preliminaries}\label{sec:preliminaries}
Let \(\Sigma\) be a finite alphabet.  We write \(\Sigma^{\infty}\) for the
set of right-infinite words over \(\Sigma\).
Let \(\epsilon\) be the empty word.
Let \(\mathbb{N}_1\) denote the set of all positive integers.
Let \(\mathbb{Q}\) denote the set of all rational numbers.

A positive integer $p\leq |u|$ is a \emph{period} of a nonempty finite
word $u$ if $u[i]=u[i+p]$ for all $1\leq i\leq |u|-p$.
Let $\mper(u)$ be the least such integer, and put
$\order(u)=|u|/\mper(u)$.

\begin{definition}
Let \(t\) be a finite word over \(\Sigma\), and let
\(n=\vert t\vert\).  The \emph{reversal} of \(t\), denoted by
\(t^R\), is defined by \(\epsilon^R=\epsilon\) when \(n=0\), and by
\[
    t^R=t[n]t[n-1]\cdots t[1]
\]
when \(n\geq 1\).  Equivalently,
\(t^R[i]=t[n+1-i]\) for every \(i\in\{1,2,\ldots,n\}\).
\end{definition}

We define \(\Pal^+=\{w\in\Sigma^+\mid w=w^R\}\) and
\(\Pal=\Pal^+\cup\{\epsilon\}\).

Note that \(\Pal^+\subseteq \Sigma^+\) is the set of all non-empty palindromes.

Given \(u\in\Sigma^+\), let \(\PL(u)=k\in\mathbb{N}_1\) be the minimal value such that \(u\) is a concatenation of \(k\) nonempty palindromes. We call \(\PL(u)\) the \emph{palindromic length} of \(u\).

Given \(u\in\Sigma^*\cup\Sigma^{\infty}\), 
let \[\Factor(u)=\{x\in\Sigma^*\mid \exists y_1\in\Sigma^*\mbox{ and }y_2\in\Sigma^*\cup\Sigma^{\infty}: u=y_1xy_2\}\mbox{, }\]
let \[\Prefix(u)=\{x\in\Sigma^*\mid \exists y\in\Sigma^*\cup\Sigma^{\infty}: u=xy\}.\]
Given \(u\in\Sigma^*\), 
let \[\Suffix(u)=\{y\in\Sigma^*\mid \exists x\in\Sigma^*: u=xy\}.\]

Given \(x,y\in\Sigma^+\), let 
\[\widetilde\occur(x,y)=\{(\alpha,\beta)\in\mathbb{N}_1^2\mid x[\alpha,\beta]=y\}.\]
\[\occur(x,y)=\vert\widetilde \occur(x,y)\vert.\]

Given \(\alpha,\beta,i\in\mathbb{N}_1\) and \(\alpha\leq i\leq \beta\), let 
\(\Mirror(\alpha,\beta,i)=\overline i\in\mathbb{N}_1\) be such that \(i-\alpha=\beta-\overline i\); equivalently,
\(\Mirror(\alpha,\beta,i)=\alpha+\beta-i\).

For a finite word \(u\), the notation \(u[a,b]\) denotes the factor
at positions \(a,a+1,\ldots,b\), with
\(1\leq a\leq b\leq\vert u\vert\).

Given \(\mu\in\mathbb N_1\) and \(u\in\Sigma^+\), define
\[
\begin{split}
\PalFac(u,\mu)=\{\,&(d_1,\ldots,d_\mu)\in\mathbb N_1^\mu:\\
 &0=d_0<d_1<\cdots<d_\mu=\vert u\vert,\\
 &u[d_{i-1}+1,d_i]\in\Pal^+
       \quad(1\leq i\leq\mu)\,\}.
\end{split}
\]
Here \(d_0=0\) is an auxiliary cut, not a component of the vector.
Thus the associated palindromic factorization is
\[
 u=u[1,d_1]u[d_1+1,d_2]\cdots u[d_{\mu-1}+1,d_\mu].
\]
Let \(\PalFac(u)=\bigcup_{\mu\geq1}\PalFac(u,\mu)\).
The notation \(\vert\vec d\vert\) denotes the number of components of
\(\vec d\), and similarly for \(\vec\pi\).

\section{Palindromic concatenation}\label{sec:palindromic-concatenation}

Given \(t\in\Sigma^+\) and \(\theta\in\mathbb{N}_1\), let \(\widehat\Gamma(t,\theta)\) be the set of all pairs \((u,\vec{\pi})\) such that 
\(u\in\Sigma^+\), \(\vec{\pi}=(\pi_1,\pi_2,\dots,\pi_{\theta})\in\mathbb{N}_1^{\theta}\), 
\[
    t\in\Prefix(u)\cap\Suffix(u),
\]

\[
    1=\pi_1<\pi_2<\cdots<\pi_{\theta}
      =\vert u\vert-\vert t\vert+1,
\]
\[
    u[\pi_i,\pi_i+\vert t\vert-1]\in\{t,t^R\}
    \qquad(i\in\{1,2,\ldots,\theta\}),
\]
and
\[
    u[\pi_i,\pi_{i+1}+\vert t\vert-1]\in\Pal^+
    \qquad(i\in\{1,2,\ldots,\theta-1\}).
\]
The last condition is vacuous when \(\theta=1\).
Let \(\widehat\Gamma(t)=\bigcup_{\theta\geq 1}\widehat\Gamma(t,\theta)\), 
let \(\Gamma(t,\theta)=\{u\in\Sigma^+:\exists\vec\pi,\ (u,\vec\pi)\in\widehat\Gamma(t,\theta)\}\), and 
let \(\Gamma(t)=\bigcup_{\theta\geq 1}\Gamma(t,\theta)\).

\begin{proposition}
\label{duif9b445689e}
If \(t\in\Sigma^+\), \((u,\vec{\pi})\in\widehat\Gamma(t)\), \(x\in\Sigma^+\), and
\(\vert x\vert\leq\vert t\vert\), then
\[
    \occur(t,x)-\occur(t,x^R)
    =\occur(u,x)-\occur(u,x^R).
\]
\end{proposition}

\begin{proof}
Put \(n=\vert t\vert\), \(m=\vert x\vert\), and
\(\theta=\vert\vec\pi\vert\).  We use the given witness
\(\vec\pi=(\pi_1,\ldots,\pi_\theta)\); thus
\((u,\vec\pi)\in\widehat\Gamma(t,\theta)\).
For
\(i\in\{1,2,\ldots,\theta\}\), define
\[
    s_i=u[\pi_i,\pi_i+n-1]
    \qquad\text{and}\qquad
    w_i=u[1,\pi_i+n-1].
\]
Thus \(s_i\in\{t,t^R\}\) for every \(i\).  Since
\(\pi_1=1\), \(t\in\Prefix(u)\), and
\(\pi_\theta=\vert u\vert-n+1\), we have
\begin{equation}
\label{eq:duif9b445689e-endpoints}
    w_1=s_1=t,
    \qquad w_\theta=u,
    \qquad s_\theta=t.
\end{equation}

For \(i\in\{1,2,\ldots,\theta-1\}\), put
\[
    q_i=u[\pi_i,\pi_{i+1}+n-1].
\]
By \((u,\vec\pi)\in\widehat\Gamma(t,\theta)\), the word \(q_i\) is a
palindrome.  Its prefix of length \(n\) is \(s_i\), whereas its
suffix of length \(n\) is \(s_{i+1}\).  In a palindrome, a suffix
of a fixed length is the reversal of the prefix of the same length.
Consequently,
\begin{equation}
\label{eq:duif9b445689e-alternation}
    s_{i+1}=s_i^R
    \qquad(i\in\{1,2,\ldots,\theta-1\}).
\end{equation}
Using \eqref{eq:duif9b445689e-alternation} and \(s_1=t\), induction gives
\begin{equation}
\label{eq:duif9b445689e-si}
    s_i=\begin{cases}
        t,   &\text{if \(i\) is odd},\\
        t^R, &\text{if \(i\) is even}.
    \end{cases}
\end{equation}
This remains valid when \(t=t^R\), in which case the two displayed
values coincide.

If \(\theta=1\), the endpoint condition gives \(u=t\), and the
assertion is immediate.  We may therefore assume \(\theta\geq2\).

We next derive the occurrence-count recurrence.  Fix
\(i\in\{1,2,\ldots,\theta-1\}\) and
\(y\in\{x,x^R\}\).  In the prefix \(w_{i+1}\), the word \(w_i\)
is a prefix, the word \(q_i\) is a suffix, and their overlap is
exactly
\[
    u[\pi_i,\pi_i+n-1]=s_i.
\]
Every occurrence of \(y\) in \(w_{i+1}\) is contained in \(w_i\)
or in \(q_i\).  Indeed, an occurrence contained in neither would,
in the coordinates of \(u\), start at or before \(\pi_i-1\) and
end at or after \(\pi_i+n\).  Its length would therefore be at least
\[
    (\pi_i+n)-(\pi_i-1)+1=n+2,
\]
contrary to
\[
    \vert y\vert=m\leq n.
\]
Furthermore, an occurrence is contained in both \(w_i\) and \(q_i\)
if and only if it is contained in their common factor \(s_i\).
Hence inclusion--exclusion gives
\begin{equation}
\label{eq:duif9b445689e-occurrence-recurrence}
    \occur(w_{i+1},y)
      =\occur(w_i,y)+\occur(q_i,y)-\occur(s_i,y).
\end{equation}

For a nonempty finite word \(v\), write
\[
    D(v)=\occur(v,x)-\occur(v,x^R).
\]
If \(v\) is a palindrome, reflection of occurrence intervals in the
centre of \(v\),
\[
    (\alpha,\beta)
      \longmapsto
    (\vert v\vert+1-\beta,\vert v\vert+1-\alpha),
\]
is a bijection between the occurrences of \(x\) and those of
\(x^R\).  Thus \(D(v)=0\) for every palindrome \(v\), and in
particular
\[
    D(q_i)=0.
\]
Subtracting \eqref{eq:duif9b445689e-occurrence-recurrence} for
\(y=x^R\) from the same equality for \(y=x\) yields
\begin{equation}
\label{eq:duif9b445689e-D-recurrence}
    D(w_{i+1})=D(w_i)-D(s_i).
\end{equation}

Reversal of occurrence intervals gives
\[
    \occur(t^R,x)=\occur(t,x^R)
    \qquad\text{and}\qquad
    \occur(t^R,x^R)=\occur(t,x),
\]
and therefore
\begin{equation}
\label{eq:duif9b445689e-reversal-sign}
    D(t^R)=-D(t).
\end{equation}
We now prove by induction on \(j\in\{1,2,\ldots,\theta\}\) that
\begin{equation}
\label{eq:duif9b445689e-prefix-values}
    D(w_j)=\begin{cases}
        D(t),&\text{if \(j\) is odd},\\
        0,   &\text{if \(j\) is even}.
    \end{cases}
\end{equation}
For \(j=1\), this follows from \(w_1=t\).  Suppose that
\(j<\theta\) and that the assertion holds for \(j\).  If \(j\) is
odd, then \(s_j=t\), and
\[
    D(w_{j+1})=D(w_j)-D(s_j)=D(t)-D(t)=0.
\]
If \(j\) is even, then \(s_j=t^R\), and
\[
    D(w_{j+1})=D(w_j)-D(s_j)
      =0-D(t^R)=D(t).
\]
This proves \eqref{eq:duif9b445689e-prefix-values}.

It remains to take \(j=\theta\).  If \(\theta\) is odd, then
\eqref{eq:duif9b445689e-prefix-values} gives
\(D(w_\theta)=D(t)\).  If \(\theta\) is even, then
\eqref{eq:duif9b445689e-si} gives \(s_\theta=t^R\), whereas
\eqref{eq:duif9b445689e-endpoints} gives \(s_\theta=t\).  Hence
\(t=t^R\).  In that case \(t\) is a palindrome, so \(D(t)=0\), and
\eqref{eq:duif9b445689e-prefix-values} again gives
\(D(w_\theta)=D(t)\).  Finally, \(w_\theta=u\), and therefore
\[
    \occur(u,x)-\occur(u,x^R)
      =D(u)=D(t)
      =\occur(t,x)-\occur(t,x^R).
\]
This is the desired identity.
\end{proof}

\section{Tailing}\label{sec:tailing}

The reflection construction in this section follows occurrences of a pattern
and its reversal across a fixed palindromic factorization.  The states whose
occurrences cross a factorization boundary will provide the cuts used in
Section~\ref{sec:trimming}.

Throughout this section, fix
\[
 t\in\Sigma^+\setminus\Pal^+,
 \qquad (u,\vec\pi)\in\widehat\Gamma(t,\theta),
 \qquad \vec d\in\PalFac(u,\mu),
\]
where \(\theta,\mu\in\mathbb N_1\),
\(\vec\pi=(\pi_1,\ldots,\pi_\theta)\), and
\(\vec d=(d_1,\ldots,d_\mu)\).  Set \(d_0=0\) and
\(\ell=\vert t\vert\).  In particular,
\[
 u[d_\delta+1,d_{\delta+1}]\in\Pal^+
 \qquad(0\leq\delta<\mu).
\]
Define
\[
\begin{split}
\Delta(t,u,\vec d)=\{\,&(\gamma,\delta,z):
       1\leq\gamma\leq\vert u\vert-\ell+1,\quad z\in\{t,t^R\},\\
 &u[\gamma,\gamma+\ell-1]=z,\quad 0\leq\delta<\mu,\\
 &d_\delta+\tfrac12<\gamma+\ell-1
             <d_{\delta+1}+\tfrac12\,\},
\end{split}
\]
where \(\gamma,\delta\) are integers, and define
\[
\begin{split}
\Delta_0(t,u,\vec d,\vec\pi)
 =\{\,&(\gamma,\delta,z)\in\Delta(t,u,\vec d):\\
      &\gamma\in\{\pi_1,\ldots,\pi_\theta\}\,\}.
\end{split}
\]
The right endpoint of an occurrence determines its second coordinate
\(\delta\) uniquely: for \(e=\gamma+\ell-1\), the condition is
\(d_\delta<e\leq d_{\delta+1}\).
For the fixed data, abbreviate
\[
 \Delta=\Delta(t,u,\vec d),
 \qquad \Delta_0=\Delta_0(t,u,\vec d,\vec\pi).
\]
The data \(t,u,\vec d,\vec\pi\) are also held fixed in the maps
\(\red_1,\red_2\) and in the iteration below; their dependence is
suppressed in the notation.

Define \(\red_1:\Delta\to\Delta\) as follows.  If
\(x=(\gamma,\delta,z)\in\Delta_0\), put \(\red_1(x)=x\).
Otherwise there is a unique \(j\in\{1,\ldots,\theta-1\}\) such that
\(\pi_j<\gamma<\pi_{j+1}\).  Put
\[
 \gamma'=\pi_j+\pi_{j+1}-\gamma,
 \qquad z'=z^R,
\]
and choose the unique \(\delta'\in\{0,\ldots,\mu-1\}\) for which
\(d_{\delta'}<\gamma'+\ell-1\leq d_{\delta'+1}\).
Set \(\red_1(x)=(\gamma',\delta',z')\).
This is equivalently the reflection rule
\[
 \gamma'+\ell-1
   =\Mirror(\pi_j,\pi_{j+1}+\ell-1,\gamma).
\]
Reflection in the palindrome
\(u[\pi_j,\pi_{j+1}+\ell-1]\) gives the factor \(z^R\) at
\(\gamma'\), and \(\pi_j<\gamma'<\pi_{j+1}\), so the resulting
triple belongs to \(\Delta\setminus\Delta_0\).

Define \(\red_2:\Delta\to\Delta\) as follows.  If
\(x=(\gamma,\delta,z)\) satisfies
\(\gamma<d_\delta+\tfrac12\), put \(\red_2(x)=x\).
Otherwise put
\[
 a=d_\delta+1,\qquad b=d_{\delta+1},\qquad
 \gamma'=a+b-\gamma-\ell+1
\]
and set
\[
 \red_2(x)=(\gamma',\delta,z^R).
\]
Equivalently, \(\gamma'+\ell-1=\Mirror(a,b,\gamma)\).
In this case the occurrence at \(\gamma\) is contained in the
palindromic block \(u[a,b]\).  Its reflected occurrence remains in
the same block, with the same second coordinate \(\delta\).

Given an initial state
\(x_1=(\gamma_1,\delta_1,z_1)\in\Delta\), let
\(x_i=(\gamma_i,\delta_i,z_i)\) be defined recursively by
\[
 x_{i+1}=\begin{cases}
 x_i,&\text{if }i>1\text{ and }\gamma_i<d_{\delta_i}+\tfrac12,\\
 \red_1(x_i),&\text{otherwise, if }i\text{ is odd},\\
 \red_2(x_i),&\text{otherwise}.
 \end{cases}
\]
In every use of this recursion the initial state, as well as the fixed
data \(t,u,\vec d,\vec\pi\), is specified.

\begin{lemma}
\label{fuje9fk7d99t}
If \((\gamma_1,\delta_1,z_1)\in\Delta(t, u, \vec{d})\setminus\Delta_0(t, u, \vec{d},\vec{\pi})\),
\(\gamma_1<d_{\delta_1}+\frac{1}{2}\), then there is
\(i_0\in\mathbb{N}_1\) such that
\[
    (\gamma_i,\delta_i,z_i)
    =(\gamma_{i_0},\delta_{i_0},z_{i_0})
    \qquad\text{for every }i\geq i_0.
\]
\end{lemma}

\begin{proof}
Use the fixed data and the abbreviations \(\Delta,\Delta_0\)
introduced above, and write \(x_i=(\gamma_i,\delta_i,z_i)\).
Define
\[
    \mathcal C
    =\{(\gamma,\delta,z)\in\Delta
       \mid \gamma<d_\delta+\tfrac12\}.
\]
Thus \(x_1\in\mathcal C\).  By the first clause in the definition of
our iteration, once \(x_j\in\mathcal C\) for some \(j>1\), we have
\(x_{j+1}=x_j\).  The same clause then applies at every subsequent
index, and therefore
\begin{equation}
\label{eq:fuje9fk7d99t-absorbing}
    x_i=x_j\qquad(i\geq j).
\end{equation}
It remains to prove that the sequence reaches \(\mathcal C\) at an
index larger than \(1\).

We first prove that \(\red_1\) and \(\red_2\) are involutions of the
finite set \(\Delta\).  Put
\[
    \ell=\vert t\vert.
\]
The map \(\red_1\) fixes every element of \(\Delta_0\).  Now let
\[
    x=(\gamma,\delta,z)\in\Delta\setminus\Delta_0
    \qquad\text{and}\qquad
    e=\gamma+\ell-1.
\]
If \(\theta=1\), then \(u=t\), every occurrence start in
\(\Delta\) equals \(\pi_1=1\), and \(\Delta=\Delta_0\).
Thus the lemma has no admissible initial state in that case.  In the
case under consideration \(\theta\geq2\).  Since
\[
    1=\pi_1<\pi_2<\cdots<\pi_\theta
      =\vert u\vert-\ell+1
\]
and \(1\leq\gamma\leq\vert u\vert-\ell+1\), there is a unique
\(j\in\{1,2,\ldots,\theta-1\}\) such that
\[
    \pi_j<\gamma<\pi_{j+1}.
\]
Set
\[
    a=\pi_j
    \qquad\text{and}\qquad
    b=\pi_{j+1}+\ell-1.
\]
Then \([\gamma,e]\subseteq[a,b]\), and the given witness
\((u,\vec\pi)\in\widehat\Gamma(t,\theta)\) gives
\[
    u[a,b]
      =u[\pi_j,\pi_{j+1}+\ell-1]\in\Pal^+.
\]
Reflection in this palindrome sends \([\gamma,e]\) to
\[
    [\gamma',e']=[a+b-e,\,a+b-\gamma].
\]
This is exactly the interval used in the definition of \(\red_1(x)\):
indeed, \(e'=\Mirror(a,b,\gamma)\) and
\(\gamma'=e'-\ell+1\).  Moreover,
\begin{equation}
\label{eq:fuje9fk7d99t-red1-coordinate}
    \gamma'
      =a+b-e
      =\pi_j+\pi_{j+1}-\gamma.
\end{equation}
The strict inequalities \(\pi_j<\gamma<\pi_{j+1}\) and
\eqref{eq:fuje9fk7d99t-red1-coordinate} imply
\[
    \pi_j<\gamma'<\pi_{j+1}.
\]
Hence \(\red_1(x)\notin\Delta_0\), and the same index \(j\) is used
when \(\red_1\) is applied a second time.  A second reflection in the
same palindrome returns \([\gamma,e]\), while
\((z^R)^R=z\).  Since the second coordinate \(\delta\) of an element of
\(\Delta\) is uniquely determined by its right endpoint, it also
returns to its original value.  Thus
\[
    \red_1^2(x)=x
    \qquad(x\in\Delta\setminus\Delta_0).
\]
Together with the identity action on \(\Delta_0\), this proves
\begin{equation}
\label{eq:fuje9fk7d99t-red1-involution}
    \red_1^2=\operatorname{id}_{\Delta}.
\end{equation}

We next consider \(\red_2\).  By definition, it fixes every element of
\(\mathcal C\).  Let
\[
    x=(\gamma,\delta,z)\in\Delta\setminus\mathcal C,
    \qquad e=\gamma+\ell-1.
\]
The definition of \(\Delta\) and the integrality of \(e\) give
\[
    d_\delta+1\leq e\leq d_{\delta+1}.
\]
Since \(x\notin\mathcal C\), the integrality of \(\gamma\) gives
\(\gamma\geq d_\delta+1\).  Consequently,
\[
    [\gamma,e]\subseteq[d_\delta+1,d_{\delta+1}].
\]
Put \(a=d_\delta+1\) and \(b=d_{\delta+1}\).  The word \(u[a,b]\)
is a palindrome by the assumed palindromic factorization of \(u\).
Reflection in this palindrome sends \([\gamma,e]\) to
\[
    [\gamma',e']=[a+b-e,\,a+b-\gamma].
\]
In particular,
\[
    a\leq\gamma'\leq e'\leq b,
\]
so \(\gamma'\geq d_\delta+1>d_\delta+\tfrac12\).  Thus
\(\red_2(x)\in\Delta\setminus\mathcal C\), with the same value of
\(\delta\).  Reflecting once more in \(u[a,b]\) gives
\[
    a+b-e'=\gamma,
    \qquad
    a+b-\gamma'=e.
\]
Together with \((z^R)^R=z\), this proves
\begin{equation}
\label{eq:fuje9fk7d99t-red2-involution}
    \red_2^2=\operatorname{id}_{\Delta}.
\end{equation}
Therefore both reductions are bijections of \(\Delta\).

Consider the finite augmented state space
\[
    \Omega=\Delta\times\{1,2\}
\]
and define \(F:\Omega\to\Omega\) by
\[
    F(x,1)=(\red_1(x),2),
    \qquad
    F(x,2)=(\red_2(x),1).
\]
Equations \eqref{eq:fuje9fk7d99t-red1-involution} and
\eqref{eq:fuje9fk7d99t-red2-involution} show that \(F\) is a
bijection; explicitly,
\[
    F^{-1}(x,1)=(\red_2(x),2),
    \qquad
    F^{-1}(x,2)=(\red_1(x),1).
\]
Since \(\Omega\) is finite, the orbit of \((x_1,1)\) under the
permutation \(F\) returns to its initial state.  Hence there is
\(q\in\mathbb N_1\) such that
\begin{equation}
\label{eq:fuje9fk7d99t-return}
    F^q(x_1,1)=(x_1,1).
\end{equation}
The second coordinate changes at every application of \(F\), so
\(q\) is even.

Let \((y_i)_{i\geq1}\) be the auxiliary sequence obtained by
alternating \(\red_1\) and \(\red_2\) without the absorbing first
clause.  Thus \(y_1=x_1\), and the augmented state corresponding to
\(y_i\) is \(F^{i-1}(x_1,1)\).  Equation
\eqref{eq:fuje9fk7d99t-return} gives
\[
    y_{q+1}=x_1\in\mathcal C.
\]
Therefore the set
\[
    \{i\in\{2,3,\ldots,q+1\}\mid y_i\in\mathcal C\}
\]
is nonempty.  Let \(j\) be its least element.  For every
\(i\in\{2,\ldots,j-1\}\), we have \(y_i\notin\mathcal C\).  At
\(i=1\), the absorbing clause is excluded by the condition \(i>1\),
and at every subsequent step before \(j\), it is excluded because the
current state is not in \(\mathcal C\).  It follows inductively that
\[
    x_i=y_i\qquad(1\leq i\leq j).
\]
In particular, \(x_j=y_j\in\mathcal C\) and \(j>1\).  Applying
\eqref{eq:fuje9fk7d99t-absorbing} and taking \(i_0=j\), we obtain
\[
    x_i=x_{i_0}\qquad(i\geq i_0).
\]
This is exactly the asserted eventual constancy.
\end{proof}

Given \(z\in\{t,t^R\}\), define
\[
\begin{split}
\Delta_T(t,u,\vec d,\vec\pi,z)=\{\,&(\gamma,\delta,z')\in
                       \Delta:\\
 &z'=z,\quad \gamma<d_\delta+\tfrac12\,\}.
\end{split}
\]
For each \(x\in\Delta\), let \((x_j^{(x)})_{j\geq1}\) denote the
iteration above with \(x_1^{(x)}=x\), using the same fixed data.
Define
\[
\begin{split}
\Delta_S(t,u,\vec d,\vec\pi,z)=\{\,&x\in
                \Delta_T(t,u,\vec d,\vec\pi,z):\\
 &x_j^{(x)}\in\Delta_0\text{ for some }j\geq1\,\}.
\end{split}
\]
Thus membership in \(\Delta_S\) refers to the orbit starting at the
particular element being tested, rather than to one common orbit.

For any occurrence \(u[\gamma,\gamma+\vert z\vert-1]=z\) of a
nonempty word \(z\) in a finite word \(u\), define
\[
 \cut(u,\gamma,z)
   =\bigl(u[1,\gamma+\vert z\vert-1],\ u[\gamma,\vert u\vert]\bigr).
\]
The position \(\gamma\) is part of the input: different occurrences of
\(z\) need not give the same pair.  This operation returns two words,
not palindromic-factorization vectors.  The two displayed intervals
overlap exactly in the chosen occurrence of \(z\).

\begin{proposition}
\label{ooid9873bw3w}
Suppose
\[
 x_1=(\gamma_1,\delta_1,z_1)
      \in\Delta_S(t,u,\vec d,\vec\pi,t),
\]
and let
\[
 (v_1,v_2)=\cut(u,\gamma_1,z_1).
\]
Then
\[
 v_1\in\Gamma(t)\qquad\text{and}\qquad v_2\in\Gamma(t).
\]
\end{proposition}

\begin{proof}
Use the fixed data \(t,u,\vec d,\vec\pi\) and the abbreviations
\(\Delta,\Delta_0\) from this section.  Put \(\ell=\vert t\vert\),
\(\theta=\vert\vec\pi\vert\), and let
\(x_i=(\gamma_i,\delta_i,z_i)\) denote the iteration starting at
\(x_1\).  By the definitions of \(\Delta_S\) and \(\Delta_T\),
\[
 z_1=t,\qquad \gamma_1<d_{\delta_1}+\tfrac12.
\]
In particular,
\[
 v_1=u[1,\gamma_1+\ell-1],\qquad
 v_2=u[\gamma_1,\vert u\vert].
\]
Membership in \(\Delta_S\) provides a visit to \(\Delta_0\), so
\[
 j_0=\min\{j\in\mathbb N_1:x_j\in\Delta_0\}
\]
is well-defined.  The definition of \(\Delta_T\) allows
\(x_1\in\Delta_0\); hence \(j_0=1\) must be treated separately.

\medskip
\noindent\textbf{Case 1: \(j_0=1\).}
There is a unique \(r\in\{1,\ldots,\theta\}\) such that
\(\gamma_1=\pi_r\).  Since \(z_1=t\), the distinguished occurrence
at \(\pi_r\) is \(t\), not \(t^R\).  Thus
\[
 v_1=u[1,\pi_r+\ell-1],\qquad
 v_2=u[\pi_r,\vert u\vert].
\]
Define
\[
 \vec\pi^{(1)}=(\pi_1,\ldots,\pi_r),\qquad
 \vec\pi^{(2)}=(\pi_j-\pi_r+1)_{j=r}^{\theta}.
\]
The vector \(\vec\pi^{(1)}\) is strictly increasing, begins at \(1\),
and ends at \(\pi_r=\vert v_1\vert-\ell+1\).  Each of its
occurrences carries \(t\) or \(t^R\), and for \(1\leq j<r\),
\[
 v_1[\pi_j,\pi_{j+1}+\ell-1]
   =u[\pi_j,\pi_{j+1}+\ell-1]\in\Pal^+.
\]
The first and last occurrences both carry \(t\).  Consequently,
\[
 (v_1,\vec\pi^{(1)})\in\widehat\Gamma(t,r).
\]
For the second vector, put \(q_j=\pi_j-\pi_r+1\)
for \(r\leq j\leq\theta\).  These positions are strictly increasing,
with
\[
 q_r=1,\qquad q_\theta=\vert v_2\vert-\ell+1.
\]
For every \(r\leq j\leq\theta\),
\[
 v_2[q_j,q_j+\ell-1]=u[\pi_j,\pi_j+\ell-1]\in\{t,t^R\},
\]
and for \(r\leq j<\theta\),
\[
 v_2[q_j,q_{j+1}+\ell-1]
   =u[\pi_j,\pi_{j+1}+\ell-1]\in\Pal^+.
\]
Again the first and last occurrences both carry \(t\), so
\[
 (v_2,\vec\pi^{(2)})\in\widehat\Gamma(t,\theta-r+1).
\]
Projecting these two memberships to \(\Gamma(t)\) proves both
conclusions in this case.  The endpoint cases are included: if
\(r=1\), then \(v_1=t\) has witness \((1)\); if \(r=\theta\),
then \(v_2=t\) has witness \((1)\).  When \(\theta=1\), both
words are \(t\), and the bridge conditions are vacuous.

\medskip
\noindent\textbf{Case 2: \(j_0>1\).}
Now \(x_1\notin\Delta_0\).  Thus the hypothesis of
Lemma~\ref{fuje9fk7d99t} is satisfied.  The visit to \(\Delta_0\)
is supplied by membership in \(\Delta_S\), not by eventual constancy
alone.  We have \(j_0\geq2\) and \(x_i\notin\Delta_0\) for all
\(1\leq i<j_0\).

We first record the structure of the finite orbit segment
\[
    x_1,x_2,\ldots,x_{j_0}.
\]
For every \(i\in\{2,\ldots,j_0-1\}\), one has
\begin{equation}
\label{eq:ooid9873bw3w-no-absorption}
    \gamma_i\geq d_{\delta_i}+\frac12.
\end{equation}
Indeed, if the reverse strict inequality held, then the first clause in
the definition of the iteration would give \(x_{i+1}=x_i\).  The same
clause would then apply at every later index, so the sequence would be
constant outside \(\Delta_0\), contrary to the definition of \(j_0\).

The map \(\red_1\) sends \(\Delta\setminus\Delta_0\) into itself.  In
fact, if
\[
    \pi_k<\gamma<\pi_{k+1},
\]
then reflection in the palindrome
\[
    u[\pi_k,\pi_{k+1}+\ell-1]
\]
sends the occurrence beginning at \(\gamma\) to the occurrence
beginning at
\[
    \gamma'=\pi_k+\pi_{k+1}-\gamma.
\]
Consequently,
\[
    \pi_k<\gamma'<\pi_{k+1},
\]
so the image is not in \(\Delta_0\).  Hence the first entrance into
\(\Delta_0\) cannot occur after an application of \(\red_1\).  Since
\(\red_1\) is applied when the transition index is odd, \(j_0-1\) is
even, and therefore
\begin{equation}
\label{eq:ooid9873bw3w-j0-odd}
    j_0\text{ is odd}.
\end{equation}

Every transition before \(j_0\) is a genuine reflection.  At an odd
transition index this follows from \(x_i\notin\Delta_0\); at an even
transition index it follows from
\eqref{eq:ooid9873bw3w-no-absorption}.  Both kinds of reflection replace
the carried word by its reversal.  Since \(z_1=t\), induction gives
\begin{equation}
\label{eq:ooid9873bw3w-alternating-z}
    z_i=
    \begin{cases}
        t,   &\text{if \(i\) is odd},\\
        t^R, &\text{if \(i\) is even}
    \end{cases}
    \qquad(1\leq i\leq j_0).
\end{equation}
In particular, by \eqref{eq:ooid9873bw3w-j0-odd},
\begin{equation}
\label{eq:ooid9873bw3w-last-z}
    z_{j_0}=t.
\end{equation}

For two distinct starting positions \(a,b\) of occurrences of words in
\(\{t,t^R\}\), write
\[
    a\sim b
\]
if
\begin{equation}
\label{eq:ooid9873bw3w-bridge-relation}
    u[\min\{a,b\},\max\{a,b\}+\ell-1]\in\Pal^+.
\end{equation}
This relation is symmetric.  Whenever \(a\sim b\), the two endpoint
occurrences carry mutually reversed words, since they are the prefix
and suffix of length \(\ell\) of the displayed palindrome.  As
\(t\neq t^R\), every chain of such bridges alternates between the two
words.

If one occurrence is obtained from the other by reflection in a
palindrome, then the smallest interval containing the two occurrences
is invariant under that reflection.  Its factor is therefore a
palindrome.  It follows from the definitions of \(\red_1\) and
\(\red_2\) that
\begin{equation}
\label{eq:ooid9873bw3w-orbit-bridges}
    \gamma_i\sim\gamma_{i+1}
    \qquad(1\leq i<j_0).
\end{equation}
The two positions in each relation are distinct.  Otherwise the same
occurrence would carry both \(z_i\) and \(z_i^R\), which is impossible
because \(t\notin\Pal^+\).

We next prove a straightening claim.  Call a finite sequence
\[
    r_0,r_1,\ldots,r_s
\]
of starting positions an \emph{admissible chain} if every position
carries an occurrence of \(t\) or \(t^R\), consecutive positions are
distinct, and
\[
    r_h\sim r_{h+1}
    \qquad(0\leq h<s).
\]
We claim that every admissible chain with distinct endpoints can be
replaced by a strictly monotone admissible chain with the same
endpoints.

To prove the claim, choose among all admissible chains with the given
endpoints one for which
\begin{equation}
\label{eq:ooid9873bw3w-total-variation}
    \sum_{h=0}^{s-1}\vert r_{h+1}-r_h\vert
\end{equation}
is minimal.  Such a chain exists by the well-ordering principle.  If it
is not strictly monotone, it contains three consecutive positions
\(a,b,c\) for which \(b\) does not lie strictly between \(a\) and
\(c\).

If \(a=c\), replace the subchain \(a,b,a\) by the single position
\(a\).  This preserves admissibility and decreases
\eqref{eq:ooid9873bw3w-total-variation}.  Suppose now that \(a\neq c\).
After interchanging \(a\) and \(c\) if necessary, we may assume
\(a<c\).  The occurrences at \(a\) and \(c\) carry the same word,
say \(w\in\{t,t^R\}\), while the occurrence at \(b\) carries
\(w^R\).

First suppose that \(a<c<b\), and put
\[
    b'=a+b-c.
\]
Reflection in the palindrome \(u[a,b+\ell-1]\) sends the occurrence
at \(c\) to an occurrence of \(w^R\) beginning at \(b'\).  Under the
same reflection, the palindrome \(u[c,b+\ell-1]\) is sent to
\(u[a,b'+\ell-1]\), so \(a\sim b'\).  Moreover, the smallest interval
containing the occurrences at \(b'\) and \(c\) is invariant under the
reflection, and hence \(b'\sim c\).  If \(b'>c\), repeat the same
construction.  At each repetition the middle position decreases by
the positive integer \(c-a\).  Eventually it is at most \(c\), while
it always remains greater than \(a\).  Equality with \(c\) is
impossible, since the occurrence there would carry both \(w\) and
\(w^R\).  Thus we obtain a position \(b^*\) with
\[
    a<b^*<c,
    \qquad a\sim b^*,
    \qquad b^*\sim c,
\]
and the occurrence at \(b^*\) carries \(w^R\).

For completeness, if \(b<a<c\), put \(b'=b+c-a\).  Reflection in
\(u[b,c+\ell-1]\) sends the occurrence at \(a\) to an occurrence of
\(w^R\) at \(b'\).  It sends the palindromic factor
\(u[b,a+\ell-1]\) to \(u[b',c+\ell-1]\), proving \(b'\sim c\).
The smallest interval containing the occurrences at \(a\) and \(b'\)
is invariant under the same reflection, so \(a\sim b'\).
If \(b'<a\), repeat this construction.  Each repetition increases the
middle position by \(c-a>0\), and the position always remains below
\(c\).  Equality with \(a\) is impossible because \(w\neq w^R\).
Thus this case also gives a position \(b^*\) strictly between \(a\)
and \(c\), with the required two bridges.  All new occurrences and
bridges lie within the interval of the palindrome used for that
reflection, so they remain factors of \(u\), even when the original
occurrences overlap.

Replacing \(b\) by \(b^*\) changes the contribution
\[
    \vert a-b\vert+\vert b-c\vert
\]
to
\[
    \vert a-b^*\vert+\vert b^*-c\vert=\vert a-c\vert,
\]
which is strictly smaller because \(b\) lies outside the interval with
endpoints \(a\) and \(c\).  This again contradicts the minimality of
\eqref{eq:ooid9873bw3w-total-variation}.  The straightening claim is
proved.

Since \(x_{j_0}\in\Delta_0\), there is an
\(r\in\{1,2,\ldots,\theta\}\) such that
\begin{equation}
\label{eq:ooid9873bw3w-pir}
    \gamma_{j_0}=\pi_r.
\end{equation}
By \eqref{eq:ooid9873bw3w-last-z}, the occurrence at \(\pi_r\) is
\(t\).

We first prove that \(v_1\in\Gamma(t)\).  Concatenate the original
chain up to \(\pi_r\) with the reversed orbit chain, omitting the
duplicate occurrence at \(\pi_r\):
\begin{equation}
\label{eq:ooid9873bw3w-left-chain}
    \pi_1,\pi_2,\ldots,\pi_r,
    \gamma_{j_0-1},\gamma_{j_0-2},\ldots,\gamma_1.
\end{equation}
This is an admissible chain.  Indeed, the first part has palindromic
bridges by \((u,\vec\pi)\in\widehat\Gamma(t,\theta)\), and the
second part has
palindromic bridges by
\eqref{eq:ooid9873bw3w-orbit-bridges}; the factors carried by the
positions agree at the junction by
\eqref{eq:ooid9873bw3w-last-z}.  Its endpoints are \(1=\pi_1\) and
\(\gamma_1\), and both endpoint occurrences are \(t\).  Since
\(x_1\notin\Delta_0\), we have \(\gamma_1\neq\pi_1=1\), and hence
\(1<\gamma_1\).  By the straightening claim,
\eqref{eq:ooid9873bw3w-left-chain} can be replaced by a strictly
increasing admissible chain
\[
    1=q_1<q_2<\cdots<q_a=\gamma_1.
\]
For every \(h\), the factor beginning at \(q_h\) is \(t\) or \(t^R\),
and for every \(h<a\),
\[
    u[q_h,q_{h+1}+\ell-1]\in\Pal^+.
\]
Furthermore, the first and last occurrences are both \(t\).  Since
\[
    v_1=u[1,\gamma_1+\ell-1]
    \qquad\text{and}\qquad
    q_a=\vert v_1\vert-\ell+1,
\]
putting \(\vec\pi^{(1)}=(q_1,\ldots,q_a)\) gives
\[
 (v_1,\vec\pi^{(1)})\in\widehat\Gamma(t,a).
\]
Indeed, all occurrence intervals and bridge intervals lie between
\(1\) and \(\gamma_1+\ell-1=\vert v_1\vert\); the endpoint
occurrences show that \(t\) is both a prefix and a suffix of \(v_1\).
Projecting to \(\Gamma(t,a)\subseteq\Gamma(t)\), we obtain
\begin{equation}
\label{eq:ooid9873bw3w-v1}
    v_1\in\Gamma(t).
\end{equation}

We now prove that \(v_2\in\Gamma(t)\).  Concatenate the orbit chain
with the remaining part of the original chain, again omitting the
duplicate occurrence at \(\pi_r\):
\begin{equation}
\label{eq:ooid9873bw3w-right-chain}
    \gamma_1,\gamma_2,\ldots,\gamma_{j_0},
    \pi_{r+1},\pi_{r+2},\ldots,\pi_\theta.
\end{equation}
This is also an admissible chain.  Its endpoints are \(\gamma_1\) and
\[
    \pi_\theta=\vert u\vert-\ell+1,
\]
and both endpoint occurrences are \(t\).  Since
\(x_1\notin\Delta_0\), one has \(\gamma_1\neq\pi_\theta\); because
\(\gamma_1\) is an occurrence start in \(u\), it follows that
\(\gamma_1<\pi_\theta\).  The straightening claim therefore gives a
strictly increasing admissible chain
\[
    \gamma_1=r_1<r_2<\cdots<r_b=\pi_\theta.
\]
Put
\[
    r_h'=r_h-\gamma_1+1
    \qquad(h\in\{1,2,\ldots,b\}).
\]
Then
\[
    r_1'=1
    \qquad\text{and}\qquad
    r_b'=\vert v_2\vert-\ell+1.
\]
Moreover,
\[
    v_2[r_h',r_h'+\ell-1]
       =u[r_h,r_h+\ell-1]\in\{t,t^R\},
\]
and, for \(h<b\),
\[
    v_2[r_h',r_{h+1}'+\ell-1]
       =u[r_h,r_{h+1}+\ell-1]\in\Pal^+.
\]
The first and last occurrences are both \(t\).  Thus, with
\(\vec\pi^{(2)}=(r_1',\ldots,r_b')\),
\[
 (v_2,\vec\pi^{(2)})\in\widehat\Gamma(t,b).
\]
Projecting to \(\Gamma(t,b)\subseteq\Gamma(t)\) gives
\begin{equation}
\label{eq:ooid9873bw3w-v2}
    v_2\in\Gamma(t).
\end{equation}
Equations \eqref{eq:ooid9873bw3w-v1} and
\eqref{eq:ooid9873bw3w-v2} complete Case~2 and the proof.
\end{proof}

\begin{lemma}
\label{uid9eurnn3n}
If \(t\in\Sigma^+\setminus\Pal^+\),
\((u,\vec\pi)\in\widehat\Gamma(t)\), and
\(\vec d\in\PalFac(u)\), then
\[
 \Delta_S(t,u,\vec d,\vec\pi,t)\neq\emptyset.
\]
\end{lemma}
\begin{proof}
Fix the data in the statement, put \(\ell=\vert t\vert\) and
\(\mu=\vert\vec d\vert\), and set \(d_0=0\).  Use the abbreviations
\(\Delta=\Delta(t,u,\vec d)\) and
\(\Delta_0=\Delta_0(t,u,\vec d,\vec\pi)\).  For
\(z\in\{t,t^R\}\), write
\[
 T(z)=\Delta_T(t,u,\vec d,\vec\pi,z),\qquad
 S(z)=\Delta_S(t,u,\vec d,\vec\pi,z).
\]
We will prove the stronger identity
\begin{equation}
\label{eq:uid9eurnn3n-surplus}
 \vert S(t)\vert=1+\vert S(t^R)\vert.
\end{equation}

\medskip
\noindent\textbf{1. Counting all cut-crossing occurrences.}
Proposition~\ref{duif9b445689e}, applied with \(x=t\), gives
\begin{equation}
\label{eq:uid9eurnn3n-occurrence-surplus}
 \occur(u,t)-\occur(u,t^R)
 =\occur(t,t)-\occur(t,t^R)=1,
\end{equation}
because \(t\neq t^R\).  For \(0\leq\delta<\mu\), put
\[
 b_\delta=u[d_\delta+1,d_{\delta+1}].
\]
Every \(b_\delta\) is a palindrome.  Reflection of occurrence
intervals inside this block is a bijection between its occurrences of
\(t\) and its occurrences of \(t^R\).  Thus
\begin{equation}
\label{eq:uid9eurnn3n-block-balance}
 \occur(b_\delta,t)=\occur(b_\delta,t^R)
 \qquad(0\leq\delta<\mu).
\end{equation}

Let \(u[\gamma,\gamma+\ell-1]=z\), with
\(z\in\{t,t^R\}\), and let \(e=\gamma+\ell-1\).
There is a unique \(\delta\) for which
\(d_\delta<e\leq d_{\delta+1}\).  This occurrence is contained
in the block \(b_\delta\) exactly when
\(\gamma\geq d_\delta+1\).  Otherwise
\(\gamma\leq d_\delta\), equivalently
\(\gamma<d_\delta+\tfrac12\), and its associated triple belongs
to \(T(z)\).
These alternatives are disjoint and exhaustive.  An occurrence
crossing several cuts is still counted just once, since its right
endpoint determines \(\delta\) uniquely.  Consequently,
\[
 \vert T(z)\vert
 =\occur(u,z)-\sum_{\delta=0}^{\mu-1}\occur(b_\delta,z).
\]
Subtracting the formulas for \(t\) and \(t^R\), and using
\eqref{eq:uid9eurnn3n-occurrence-surplus} and
\eqref{eq:uid9eurnn3n-block-balance}, yields
\begin{equation}
\label{eq:uid9eurnn3n-T-surplus}
 \vert T(t)\vert-\vert T(t^R)\vert=1.
\end{equation}
Here it is essential that the definition of \(T(z)\) includes
cut-crossing states in \(\Delta_0\).

\medskip
\noindent\textbf{2. Pairing the states whose orbits never reach
\(\Delta_0\).}
Set
\[
 \mathcal C=T(t)\cup T(t^R),\qquad
 N(z)=T(z)\setminus S(z)\quad(z\in\{t,t^R\}).
\]
The union defining \(\mathcal C\) is disjoint because \(t\neq t^R\).
If \(x\in T(z)\cap\Delta_0\), then the iteration starting at
\(x\) already visits \(\Delta_0\) at index \(1\), so
\(x\in S(z)\).  In particular,
\begin{equation}
\label{eq:uid9eurnn3n-N-outside}
 N(z)\subseteq\mathcal C\setminus\Delta_0.
\end{equation}

We use the properties of the reductions verified in the proof of
Lemma~\ref{fuje9fk7d99t}.  Both \(\red_1\) and \(\red_2\) are
involutions of \(\Delta\).  On \(\Delta\setminus\Delta_0\),
\(\red_1\) reverses the third coordinate.  On
\(\Delta\setminus\mathcal C\), \(\red_2\) also reverses the
third coordinate and maps that set into itself.  In particular,
\(\red_2\) has no fixed point outside \(\mathcal C\), since
\(t\neq t^R\).

Fix \(z\in\{t,t^R\}\) and \(x\in N(z)\), and denote its
iteration by
\[
 x_i=x_i^{(x)}\qquad(i\geq1),\qquad x_1=x.
\]
By the definition of \(N(z)\), no \(x_i\) belongs to
\(\Delta_0\).  By \eqref{eq:uid9eurnn3n-N-outside},
Lemma~\ref{fuje9fk7d99t} applies and the sequence is eventually
constant.  Its eventual value must lie in \(\mathcal C\): if it
lay outside \(\mathcal C\), an even transition index in the
constant part of the sequence would apply \(\red_2\) and change
the third coordinate, a contradiction.
It follows that the first-return index
\begin{equation}
\label{eq:uid9eurnn3n-tau}
 \tau(x)=\min\{j\geq2:x_j\in\mathcal C\}
\end{equation}
exists.  Write \(\tau=\tau(x)\).  Then
\[
 x_i\notin\mathcal C\cup\Delta_0\quad(2\leq i<\tau),
 \qquad x_\tau\in\mathcal C\setminus\Delta_0.
\]

The index \(\tau\) is even.  Otherwise \(\tau\geq3\), the
transition at index \(\tau-1\) would use \(\red_2\), and
\(x_{\tau-1}\notin\mathcal C\).  Since \(\red_2\) preserves
\(\Delta\setminus\mathcal C\), this would imply
\(x_\tau\notin\mathcal C\), a contradiction.
Before this first return the absorbing clause never applies: it is
excluded at index \(1\) by the condition \(i>1\), and at the
intermediate indices by the definition of \(\tau\).
Every one of the \(\tau-1\) transitions therefore reverses the
third coordinate.  As \(\tau-1\) is odd,
\[
 y:=x_\tau\in T(z^R).
\]
Define \(\Phi(x)=y\).

We next show that \(y\in N(z^R)\) and \(\Phi(y)=x\).
The forward segment
\[
 x=x_1,x_2,\ldots,x_\tau=y
\]
uses the reductions in an alternating order starting and ending with
\(\red_1\).  Since both reductions are involutions and the number
of transitions is odd, the reversed segment
\[
 y=x_\tau,x_{\tau-1},\ldots,x_2,x_1=x
\]
uses exactly the same alternating order.  This segment is the actual
initial orbit from \(y\), not merely a formal reversal.  Indeed, the
absorbing clause is excluded at its initial index, and every
intermediate state is outside \(\mathcal C\).  Explicitly, if
\((y_j)_{j\geq1}\) is the iteration starting at \(y\), induction
on \(j\) gives
\[
 y_j=x_{\tau-j+1}\qquad(1\leq j\leq\tau).
\]
For the induction step, the forward transition being reversed has
index \(\tau-j\), which has the same parity as \(j\) because
\(\tau\) is even.  Hence the same involution applies in both
directions.

At index \(\tau>1\), the reversed orbit reaches \(x\in\mathcal C\)
and remains constant there by the absorbing clause.  None of the
states in this orbit is in \(\Delta_0\).  Thus \(y\notin S(z^R)\),
and consequently \(y\in N(z^R)\).  Its first return to
\(\mathcal C\) after its initial state is exactly \(x\), so
\(\Phi(y)=x\).
We have constructed mutually inverse maps between \(N(t)\) and
\(N(t^R)\).  Therefore
\begin{equation}
\label{eq:uid9eurnn3n-N-balance}
 \vert N(t)\vert=\vert N(t^R)\vert.
\end{equation}

\medskip
\noindent\textbf{3. The surplus must belong to \(S(t)\).}
For \(z\in\{t,t^R\}\), the sets \(S(z)\) and \(N(z)\) form a
disjoint partition of \(T(z)\).  All these sets are finite.  Hence
\eqref{eq:uid9eurnn3n-T-surplus} and
\eqref{eq:uid9eurnn3n-N-balance} give
\[
\begin{split}
 \vert S(t)\vert-\vert S(t^R)\vert
 &=\bigl(\vert T(t)\vert-\vert T(t^R)\vert\bigr)
   -\bigl(\vert N(t)\vert-\vert N(t^R)\vert\bigr)\\
 &=1.
\end{split}
\]
This proves \eqref{eq:uid9eurnn3n-surplus}.  In particular,
\(\vert S(t)\vert\geq1\), which is the asserted nonemptiness of
\(\Delta_S(t,u,\vec d,\vec\pi,t)\).
\end{proof}

\section{Trimming and palindromic length}\label{sec:trimming}

The next lemma applies to arbitrary nonempty finite words; in particular,
its pattern \(t\) is not assumed to be nonpalindromic.

\begin{lemma}
\label{k90892bjdie}
If \(t_1,t_2,t\in\Sigma^+\), \(t=t_1t_2\), \(u\in\Gamma(t)\), and
\(\overline u\in\Sigma^+\) satisfies \(u=t_1\overline u\), then
\[
 \overline u\in\Gamma(t_2).
\]
\end{lemma}
\begin{proof}
Put \(a=\vert t_1\vert\), \(b=\vert t_2\vert\),
\(\ell=a+b\), and \(N=\vert u\vert\).  If \(u=t\), then
\(\overline u=t_2\), and the one-component vector \((1)\) witnesses
\(\overline u\in\Gamma(t_2)\).  Henceforth suppose \(N>\ell\).

\medskip
\noindent\textbf{A straightening claim.}
We first record a version of the bridge argument that permits a
palindromic pattern.  Fix any \(q\in\Sigma^+\), and write
\(r=\vert q\vert\).  For two distinct starting positions \(c,d\) of
occurrences of \(q\) or \(q^R\) in \(u\), write \(c\sim_q d\) if
\[
 u[\min\{c,d\},\max\{c,d\}+r-1]\in\Pal^+.
\]
Suppose a finite chain of such occurrence positions has the property
that each consecutive pair is either equal or related by \(\sim_q\).
Then it can be replaced by a monotone chain with the same endpoints,
strictly monotone when those endpoints are distinct.  No assumption
\(q\neq q^R\) is required.

If the endpoints are equal, keep just that position.  For distinct
endpoints, first delete any consecutive repeated positions.  Among
all chains with the specified endpoints and the required bridges,
choose one minimizing the nonnegative integer
\[
 \sum_j\vert c_{j+1}-c_j\vert.
\]
If it is not strictly monotone, it contains three consecutive
positions \(A,B,C\) with \(B\) not strictly between \(A\) and \(C\).
If \(A=C\), delete the backtrack \(A,B,A\), retaining its single
endpoint.  This strictly decreases the sum.  Otherwise, by reversing
this three-position subchain if necessary, assume \(A<C\).

If \(A<C<B\), reflect inside the palindrome
\(u[A,B+r-1]\).  The occurrence at \(C\) is sent to an occurrence
of \(q\) or \(q^R\) at
\[
 B'=A+B-C.
\]
The palindromic subfactor \(u[C,B+r-1]\) is sent to
\(u[A,B'+r-1]\), proving \(A\sim_q B'\).  The smallest interval
containing the occurrences at \(C\) and \(B'\) is invariant under
this same reflection, so it is a palindrome as well.  Thus
\(B'\sim_q C\) when \(B'\neq C\).  When \(B'=C\), delete the
repeated position instead.  All intervals just used lie inside
\([A,B+r-1]\).  Since \(A<B'<B\),
\[
 \vert A-B'\vert+\vert B'-C\vert
 <\vert A-B\vert+\vert B-C\vert.
\]

If \(B<A<C\), reflect instead inside \(u[B,C+r-1]\), and put
\[
 B'=B+C-A.
\]
The occurrence at \(A\) is reflected to \(B'\).  The palindrome
\(u[B,A+r-1]\) is reflected to \(u[B',C+r-1]\), and the hull of
the occurrences at \(A\) and \(B'\) is invariant under the same
reflection.  These facts give the bridges from \(A\) to \(B'\) and
from \(B'\) to \(C\), with any repeated consecutive position deleted.
The new intervals lie inside \([B,C+r-1]\).  Since \(B<B'<C\),
\[
 \vert A-B'\vert+\vert B'-C\vert
 <\vert A-B\vert+\vert B-C\vert.
\]
Both cases contradict minimality.  This proves the claim, including
when the occurrences overlap or \(q\) is palindromic.

\medskip
\noindent\textbf{Projecting the distinguished occurrences.}
Choose a witness
\[
 (u,\vec\pi)\in\widehat\Gamma(t,\theta),
 \qquad \vec\pi=(\pi_1,\ldots,\pi_\theta).
\]
Let \(s_j=u[\pi_j,\pi_j+\ell-1]\).  Each palindromic bridge has
\(s_j\) as its length-\(\ell\) prefix and \(s_{j+1}\) as its
length-\(\ell\) suffix.  Consequently \(s_{j+1}=s_j^R\).  Since
\(s_1=t\), it follows that
\[
 s_j=\begin{cases}
 t,&j\text{ odd},\\
 t^R,&j\text{ even}.
 \end{cases}
\]
This formula is valid also when \(t=t^R\).

For \(1\leq j\leq\theta\), select the occurrence starting at
\[
 \beta_j=\begin{cases}
 \pi_j+a,&j\text{ odd},\\
 \pi_j,&j\text{ even}.
 \end{cases}
\]
It carries \(t_2\) at an odd index, and \(t_2^R\) at an even index.
Reflection in \(u[\pi_j,\pi_{j+1}+\ell-1]\) sends the selected
length-\(b\) occurrence at \(\beta_j\) to one starting at
\[
 \pi_j+\pi_{j+1}+\ell-1-(\beta_j+b-1)
 =\pi_j+\pi_{j+1}+a-\beta_j
 =\beta_{j+1}.
\]
Thus successive selected occurrences have a palindromic bridge when
they are distinct; coincident successive positions may be deleted.
All these occurrences and bridges are factors of \(u\).
The initial position is \(\beta_1=a+1\), carrying \(t_2\).

If \(\theta\) is odd, the last position is
\[
 \beta_\theta=\pi_\theta+a=N-b+1,
\]
and it also carries \(t_2\).  If \(\theta\) is even, the final
distinguished occurrence carries both \(t^R\), by the parity formula,
and \(t\), because \(t\) is a suffix of \(u\).  Therefore
\(t=t^R\).  In this case \(u[\pi_\theta,N]=t\) is a palindrome.
Reflect its length-\(b\) prefix, at \(\beta_\theta=\pi_\theta\),
to its length-\(b\) suffix, at \(N-b+1=\pi_\theta+a\).
Append this last position to the selected chain.  The appended
occurrence carries \(t_2\), and the two occurrences have a
palindromic bridge inside this final copy of \(t\).

In either case we have a chain from \(a+1\) to \(N-b+1\), whose
endpoint occurrences both carry \(t_2\).  These endpoints are
distinct because \(N>\ell\).  Apply the straightening claim with
\(q=t_2\) to obtain
\[
 a+1=\rho_1<\rho_2<\cdots<\rho_h=N-b+1.
\]
Every position carries \(t_2\) or \(t_2^R\), and every consecutive
pair has a palindromic bridge.  In particular, every new occurrence
and bridge lies in \(u[a+1,N]=\overline u\).
Set \(\sigma_j=\rho_j-a\).  Then
\[
 \sigma_1=1,\qquad
 \sigma_h=N-a-b+1=\vert\overline u\vert-b+1.
\]
Translation of coordinates gives
\[
 \overline u[\sigma_j,\sigma_j+b-1]\in\{t_2,t_2^R\},
\]
and, for \(j<h\),
\[
 \overline u[\sigma_j,\sigma_{j+1}+b-1]\in\Pal^+.
\]
Both endpoint occurrences are \(t_2\).  Hence
\[
 (\overline u,(\sigma_1,\ldots,\sigma_h))
 \in\widehat\Gamma(t_2,h),
\]
which proves the lemma.
\end{proof}

We also require its symmetric form:
\begin{equation}
\label{eq:k90892bjdie-right-trimming}
 \begin{gathered}
 t=t_1t_2,\quad u\in\Gamma(t)
 \quad\Longrightarrow\quad
 u[1,\vert u\vert-\vert t_2\vert]\in\Gamma(t_1).
 \end{gathered}
\end{equation}
To justify it, observe first that \(W\in\Gamma(s)\) implies
\(W^R\in\Gamma(s^R)\): a witness \((\rho_1,\ldots,\rho_h)\)
is reflected to the increasing witness
\[
 (\vert W\vert-\vert s\vert+2-\rho_{h+1-j})_{j=1}^{h}.
\]
The occurrence and bridge conditions follow by reversing their
intervals.  Apply Lemma~\ref{k90892bjdie} to
\(u^R\in\Gamma(t^R)\) and \(t^R=t_2^Rt_1^R\), deleting the
prefix \(t_2^R\).  Reversing the resulting membership in
\(\Gamma(t_1^R)\) proves
\eqref{eq:k90892bjdie-right-trimming}.

\subsection{The factorization-boundary cut}

We now formalize the eight-output operation at a factorization
boundary.  It is distinguished by its arguments from the earlier
three-argument occurrence cut, whose two output words overlap.
Fix
\[
 t\in\Sigma^+\setminus\Pal^+,\qquad
 (u,\vec\pi)\in\widehat\Gamma(t),\qquad
 \vec d=(d_1,\ldots,d_\mu)\in\PalFac(u),
\]
and let \(d_0=0\).  Choose
\[
 (\gamma,\delta,z)\in\Delta_S(t,u,\vec d,\vec\pi,t).
\]
Thus \(z=t\).  The occurrence and crossing conditions imply
\begin{equation}
\label{eq:qyyu938jfur-crossing}
 \gamma\leq d_\delta<\gamma+\vert t\vert-1
 \leq d_{\delta+1}.
\end{equation}
Consequently \(1\leq\delta\leq\mu-1\), and the integer
\(r=d_\delta-\gamma+1\) satisfies \(1\leq r<\vert t\vert\).
Define
\begin{equation}
\label{eq:qyyu938jfur-cut-words}
 \begin{aligned}
 v_1&=u[1,d_\delta],&
 v_2&=u[d_\delta+1,\vert u\vert],\\
 t_1&=u[\gamma,d_\delta]=t[1,r],&
 t_2&=u[d_\delta+1,\gamma+\vert t\vert-1]
       =t[r+1,\vert t\vert].
 \end{aligned}
\end{equation}
All four words are nonempty, and \(u=v_1v_2\), \(t=t_1t_2\).
Set
\begin{equation}
\label{eq:qyyu938jfur-cut-vectors}
 \begin{aligned}
 \vec d_1&=(d_1,\ldots,d_\delta),\\
 \vec d_2&=(d_{\delta+1}-d_\delta,\ldots,d_\mu-d_\delta).
 \end{aligned}
\end{equation}
These vectors retain whole palindromic blocks and hence satisfy
\begin{equation}
\label{eq:qyyu938jfur-cut-budget}
 \begin{gathered}
 \vec d_1\in\PalFac(v_1,\delta),\qquad
 \vec d_2\in\PalFac(v_2,\mu-\delta),\\
 \vert\vec d_1\vert+\vert\vec d_2\vert=\mu.
 \end{gathered}
\end{equation}

We verify the existence of witness vectors for the new patterns.
By Proposition~\ref{ooid9873bw3w}, the earlier occurrence cut gives
\[
 \widetilde v_1=u[1,\gamma+\vert t\vert-1]\in\Gamma(t),
 \qquad
 \widetilde v_2=u[\gamma,\vert u\vert]\in\Gamma(t).
\]
Here \(\widetilde v_1=v_1t_2\) and
\(\widetilde v_2=t_1v_2\).  Lemma~\ref{k90892bjdie} and
\eqref{eq:k90892bjdie-right-trimming} therefore imply
\begin{equation}
\label{eq:qyyu938jfur-cut-gamma}
 v_1\in\Gamma(t_1),\qquad v_2\in\Gamma(t_2).
\end{equation}
Thus we may choose vectors \(\vec\pi_1,\vec\pi_2\) with
\[
 (v_i,\vec\pi_i)\in\widehat\Gamma(t_i)
 \qquad(i\in\{1,2\}).
\]
To make the operation single-valued, choose for each pair \((t_i,v_i)\)
a witness with the fewest components and, among those, the
lexicographically least one.  Such a choice exists: the set of
witnesses is nonempty by \eqref{eq:qyyu938jfur-cut-gamma} and finite,
since each is a strictly increasing vector of positions in \(v_i\).
Any fixed choice of valid witnesses would suffice in the proof below.
We now define
\begin{equation}
\label{eq:qyyu938jfur-cut-definition}
 \begin{split}
 &\cut(u,\vec d,\gamma,\delta,z)\\
 &\hspace{1em}=(v_1,v_2,t_1,t_2,
                \vec\pi_1,\vec\pi_2,\vec d_1,\vec d_2).
 \end{split}
\end{equation}
The input witness \(\vec\pi\) is used to certify membership in
\(\Delta_S\).  Once that membership holds, all eight selected
outputs above are determined by the displayed inputs.

\subsection{The recursive family}

For arbitrary \(t\in\Sigma^+\),
\((u,\vec\pi)\in\widehat\Gamma(t)\), and
\(\vec d\in\PalFac(u)\), define
\[
 \Omega(0)=\{(t,u,\vec\pi,\vec d,1)\}.
\]
A state is written in the consistent order
\[
 \omega=(s,W,\vec\rho,\vec e,g),
 \qquad (W,\vec\rho)\in\widehat\Gamma(s),\quad
 \vec e\in\PalFac(W).
\]
The coordinate \(g\) records the starting position of \(s\) in the
\emph{original pattern} \(t\), not an occurrence position in \(W\).

If \(s\in\Pal^+\), put \(\mathcal R(\omega)=\{\omega\}\).
If \(s\notin\Pal^+\), Lemma~\ref{uid9eurnn3n} gives
\[
 \Delta_S(s,W,\vec e,\vec\rho,s)\neq\emptyset.
\]
Select its element \((\gamma,\delta,s)\) with the smallest starting
position \(\gamma\); its second coordinate is then unique.  Write
\[
 \begin{split}
 &\cut(W,\vec e,\gamma,\delta,s)\\
 &\hspace{1em}=(W_1,W_2,s_1,s_2,
               \vec\rho_1,\vec\rho_2,\vec e_1,\vec e_2),
 \end{split}
\]
and put
\[
 \begin{split}
 \mathcal R(\omega)=\{
 &(s_1,W_1,\vec\rho_1,\vec e_1,g),\\
 &(s_2,W_2,\vec\rho_2,\vec e_2,g+\vert s_1\vert)\}.
 \end{split}
\]
Thus only one admissible cut is selected for each nonterminal state;
we do not collect the outputs of all possible choices.  For
\(j\in\mathbb N_1\cup\{0\}\), define
\begin{equation}
\label{eq:qyyu938jfur-Omega-recursion}
 \Omega(j+1)=\bigcup_{\omega\in\Omega(j)}\mathcal R(\omega).
\end{equation}
The memberships needed for subsequent steps follow from
\eqref{eq:qyyu938jfur-cut-budget} and
\eqref{eq:qyyu938jfur-cut-gamma}.  On palindromic labels we retain the
state without evaluating \(\Delta_S\), whose reflection construction
was introduced for nonpalindromic patterns.  The proof below shows
that distinct states have distinct offsets, so taking a union in
\eqref{eq:qyyu938jfur-Omega-recursion} does not identify different
pieces of the original pattern.

\begin{proposition}
\label{qyyu938jfur}
If \(\PL(t)=k\) and \(u\in\Gamma(t)\), then \(\PL(u)\geq k\).
\end{proposition}
\begin{proof}
Put \(\mu=\PL(u)\).  Choose
\[
 (u,\vec\pi)\in\widehat\Gamma(t),\qquad
 \vec d\in\PalFac(u,\mu),
\]
and use these data to form \(\Omega(j)\) as above.

\medskip
\noindent\textbf{1. The invariants.}
For every \(j\geq0\), we claim that the states of \(\Omega(j)\)
can be listed as
\[
 \omega_{j,i}=(s_{j,i},W_{j,i},\vec\rho_{j,i},\vec e_{j,i},g_{j,i})
 \qquad(1\leq i\leq h_j)
\]
in strictly increasing order of \(g_{j,i}\), and satisfy
\begin{equation}
\label{eq:qyyu938jfur-membership-invariant}
 (W_{j,i},\vec\rho_{j,i})\in\widehat\Gamma(s_{j,i}),
 \qquad \vec e_{j,i}\in\PalFac(W_{j,i}),
\end{equation}
\begin{equation}
\label{eq:qyyu938jfur-concatenation-invariant}
 t=s_{j,1}\cdots s_{j,h_j},\qquad
 u=W_{j,1}\cdots W_{j,h_j},
\end{equation}
\begin{equation}
\label{eq:qyyu938jfur-offset-invariant}
 g_{j,i}=1+\sum_{a<i}\vert s_{j,a}\vert,
\end{equation}
and
\begin{equation}
\label{eq:qyyu938jfur-budget-invariant}
 \sum_{i=1}^{h_j}\vert\vec e_{j,i}\vert=\mu.
\end{equation}
All patterns and words in these formulas are nonempty.

The assertions hold at \(j=0\), where \(h_0=1\).  Suppose they
hold at stage \(j\).  A retained state changes none of the formulas.
A split replaces its pattern \(s\) by two nonempty factors
\(s=s_1s_2\), and its word \(W\) by two nonempty consecutive
factors \(W=W_1W_2\).  The new memberships follow from
\eqref{eq:qyyu938jfur-cut-budget} and
\eqref{eq:qyyu938jfur-cut-gamma}.  The offsets \(g\) and
\(g+\vert s_1\vert\) assign the children the two consecutive
intervals
\[
 [g,g+\vert s_1\vert-1],\qquad
 [g+\vert s_1\vert,g+\vert s\vert-1]
\]
inside their parent's interval in \(t\).  Parent intervals are
disjoint by the induction hypothesis.  Thus all new offsets remain
distinct, and the states can be ordered as asserted.  There is no
loss of multiplicity in the union defining \(\Omega(j+1)\).
The pattern and word concatenations and the offset formula are
preserved.  Finally, the two new factorization vectors have a total
of \(\vert\vec e\vert\) components, by
\eqref{eq:qyyu938jfur-cut-budget}, so their total budget is unchanged.
This proves all four invariants by induction.

\medskip
\noindent\textbf{2. Termination.}
Every \(\vec e_{j,i}\) has at least one component.  Therefore
\eqref{eq:qyyu938jfur-budget-invariant} gives
\begin{equation}
\label{eq:qyyu938jfur-state-bound}
 h_j\leq\mu\qquad(j\geq0).
\end{equation}
Let \(b_j\) be the number of nonpalindromic patterns among the
states at stage \(j\).  Lemma~\ref{uid9eurnn3n} ensures that each
such state can be split, and the definition retains exactly the
other states.  Distinctness of the new offsets gives
\[
 h_{j+1}=h_j+b_j.
\]
If \(b_j>0\) at each of the stages \(0,\ldots,\mu-1\), then
\(h_\mu\geq1+\mu\), contradicting
\eqref{eq:qyyu938jfur-state-bound}.  Hence there is
\(j_0\in\{0,\ldots,\mu-1\}\) such that \(b_{j_0}=0\).
Every pattern in \(\Omega(j_0)\) is a nonempty palindrome.  All its
states are retained thereafter, so
\[
 \Omega(j)=\Omega(j_0)\qquad(j\geq j_0).
\]

\medskip
\noindent\textbf{3. Counting the terminal factors.}
Set \(h=h_{j_0}\).  By
\eqref{eq:qyyu938jfur-concatenation-invariant}, the terminal labels
give the palindromic factorization
\[
 t=s_{j_0,1}\cdots s_{j_0,h}.
\]
Consequently \(k=\PL(t)\leq h\).  On the other hand,
\eqref{eq:qyyu938jfur-state-bound} gives
\(h\leq\mu=\PL(u)\).  Therefore
\[
 \boxed{\PL(t)=k\leq h\leq\mu=\PL(u).}
\]
This proves the proposition.  If \(t\) is already a palindrome, the
same argument stops at \(j_0=0\) and gives \(1\leq\PL(u)\).
\end{proof}

\section{Density of prefixes with maximal palindromic length}\label{sec:density}
For a nonempty finite word \(u\), define
\[
 \maxPrefixPL(u)
 =\max\{\PL(v):v\in\Prefix(u)\cap\Sigma^+\}.
\]
Only nonempty prefixes are used, since \(\PL\) was defined on
\(\Sigma^+\).  For \(w\in\Sigma^{\infty}\) and
\(n\in\mathbb N_1\), put
\[
\begin{split}
 T_w(n)=\{\,&w[1,j]:1\leq j\leq n,\\
           &\PL(w[1,j])=\maxPrefixPL(w[1,j])\,\}.
\end{split}
\]
Thus a prefix is counted whenever its palindromic length attains the
running maximum, including a tie with an earlier maximum.  Let
\[
 \FF(w)=\Pal^+\cap\Prefix(w),\qquad
 P_w(n)=\vert\Pal^+\cap\Prefix(w[1,n])\vert.
\]

\begin{lemma}[Translation between palindromic prefixes]
\label{lem:nfh76dg347g-translation}
Let \(W\in\Sigma^+\), and suppose that
\(1\leq a<b\leq\vert W\vert\) satisfy
\[
 W[1,a],W[1,b]\in\Pal^+.
\]
Put \(d=b-a\).  For every integer \(r\) with \(1\leq r\leq a\),
\begin{equation}
\label{eq:nfh76dg347g-shift-membership}
 W[1,r+d]\in\Gamma(W[1,r]).
\end{equation}
Consequently,
\begin{equation}
\label{eq:nfh76dg347g-shift-PL}
 \PL(W[1,r+d])\geq\PL(W[1,r])
 \qquad(1\leq r\leq a).
\end{equation}
\end{lemma}
\begin{proof}
Put \(A=W[1,a]\) and \(B=W[1,b]\).  The word \(A\) is a
prefix of the palindrome \(B\), so \(A^R\) is a suffix of \(B\).
Since \(A=A^R\), both the prefix and the suffix of \(B\) of length
\(a\) equal \(A\).  Therefore
\[
 (B,(1,d+1))\in\widehat\Gamma(A,2).
\]
Indeed, the distinguished occurrences start at \(1\) and
\(d+1=b-a+1\), and their sole bridge is the whole palindrome
\(B\).  This verification does not require the two occurrences to
be disjoint.  In particular, \(B\in\Gamma(A)\).

If \(r=a\), then \(W[1,r+d]=B\) and \(W[1,r]=A\), which proves
\eqref{eq:nfh76dg347g-shift-membership} in that case.  If \(r<a\),
write
\[
 A=vs,\qquad v=W[1,r],\qquad s=W[r+1,a].
\]
Both \(v\) and \(s\) are nonempty.  The symmetric trimming property
\eqref{eq:k90892bjdie-right-trimming}, proved after
Lemma~\ref{k90892bjdie}, gives
\[
 B[1,\vert B\vert-\vert s\vert]\in\Gamma(v).
\]
Since \(\vert B\vert-\vert s\vert=b-(a-r)=r+d\), this is exactly
\eqref{eq:nfh76dg347g-shift-membership}.  Applying
Proposition~\ref{qyyu938jfur} proves
\eqref{eq:nfh76dg347g-shift-PL}.
\end{proof}

\begin{proposition}
\label{nfh76dg347g}
For every \(w\in\Sigma^{\infty}\) and every
\(n\in\mathbb N_1\),
\begin{equation}
\label{eq:nfh76dg347g-finite-bound}
 \vert T_w(n)\vert
 \geq P_w(n)
 =\vert\Pal^+\cap\Prefix(w[1,n])\vert.
\end{equation}
In particular, if \(\vert\FF(w)\vert=\infty\), then
\begin{equation}
\label{eq:nfh76dg347g-liminf}
 \liminf_{n\to\infty}
 \frac{\vert T_w(n)\vert}
      {\vert\Pal^+\cap\Prefix(w[1,n])\vert}
 \geq 1.
\end{equation}
\end{proposition}
\begin{proof}
Fix \(n\geq1\).  For \(1\leq j\leq n\), abbreviate
\[
 L(j)=\PL(w[1,j]),\qquad
 M(j)=\max_{1\leq r\leq j}L(r).
\]
Let
\[
 a_1<a_2<\cdots<a_s
\]
be all lengths of nonempty palindromic prefixes of \(w[1,n]\).
The first letter is a palindrome, so \(s\geq1\) and \(a_1=1\).
Also, \(s=P_w(n)\).  Set \(a_0=0\).

For each \(i\in\{1,\ldots,s\}\), choose the rightmost occurrence
of the maximal prefix palindromic length up to \(a_i\):
\begin{equation}
\label{eq:nfh76dg347g-rightmost-maximum}
 r_i=\max\{r\in\{1,\ldots,a_i\}:L(r)=M(a_i)\}.
\end{equation}
This set is nonempty because a maximum of a finite nonempty set of
integers is attained.  We claim that
\begin{equation}
\label{eq:nfh76dg347g-disjoint-intervals}
 a_{i-1}<r_i\leq a_i
 \qquad(1\leq i\leq s).
\end{equation}
For \(i=1\), we have \(r_1=1\), so the claim holds.
Suppose \(i\geq2\), and put \(d_i=a_i-a_{i-1}>0\).
If \(r_i\leq a_{i-1}\), apply
Lemma~\ref{lem:nfh76dg347g-translation} to the palindromic prefixes
of lengths \(a_{i-1}\) and \(a_i\), with \(r=r_i\).
It gives
\[
 L(r_i+d_i)\geq L(r_i)=M(a_i).
\]
But
\[
 r_i<r_i+d_i\leq a_i.
\]
By the definition of \(M(a_i)\), the preceding inequality must be
an equality.  This contradicts the choice of \(r_i\) as the
rightmost index attaining \(M(a_i)\).  Thus
\(r_i>a_{i-1}\), proving
\eqref{eq:nfh76dg347g-disjoint-intervals}.

For each \(i\), the definition of \(r_i\) gives
\[
 L(r_i)=M(a_i).
\]
Since \(r_i\leq a_i\), we also have
\[
 L(r_i)\leq M(r_i)\leq M(a_i)=L(r_i).
\]
Consequently \(L(r_i)=M(r_i)\), so
\[
 w[1,r_i]\in T_w(n).
\]
The intervals \((a_{i-1},a_i]\) in
\eqref{eq:nfh76dg347g-disjoint-intervals} are disjoint and ordered.
Hence \(r_1<\cdots<r_s\), and the \(s\) prefixes just constructed
are distinct.  Therefore
\[
 \vert T_w(n)\vert\geq s=P_w(n),
\]
which proves \eqref{eq:nfh76dg347g-finite-bound}.

The denominator \(P_w(n)\) is positive for every \(n\geq1\).
Dividing \eqref{eq:nfh76dg347g-finite-bound} by \(P_w(n)\) and
taking the lower limit proves \eqref{eq:nfh76dg347g-liminf}.
When \(\vert\FF(w)\vert=\infty\), one also has
\(P_w(n)\to\infty\), and hence \(\vert T_w(n)\vert\to\infty\).
\end{proof}

\begin{remark}[Scope of the limit statement]
\label{rem:nfh76dg347g-limit-scope}
The proof establishes a pointwise lower bound and therefore a positive
lower limit.  It does not establish that
\(\vert T_w(n)\vert/P_w(n)\) has an ordinary limit.
If an ordinary limit exists, whether finite or \(+\infty\),
\eqref{eq:nfh76dg347g-finite-bound} forces it to be at least \(1\).
Thus replacing \(\liminf\) in
\eqref{eq:nfh76dg347g-liminf} by \(\lim\) would require a separate
convergence argument; no such claim is made here.
\end{remark}

\begin{example}[Sharpness of the coefficient]\label{ex:sharpness}
Let \(w=(ab)^\infty\), with \(a\neq b\).  Its odd-length prefixes
are palindromes.  Every even-length prefix is nonpalindromic but is
an odd-length palindrome followed by one letter, so its palindromic
length is \(2\).  Thus the running-maximum indices are \(1\) and
all positive even integers, and
\[
 P_w(n)=\left\lceil\frac n2\right\rceil,
 \qquad
 \vert T_w(n)\vert=1+\left\lfloor\frac n2\right\rfloor.
\]
Equality holds in \eqref{eq:nfh76dg347g-finite-bound} for every odd
\(n\), and \(\vert T_w(n)\vert/P_w(n)\to1\).  Therefore the
universal constant \(1\) in \eqref{eq:nfh76dg347g-liminf} cannot
be increased.
\end{example}

\section*{Declaration of generative AI and AI-assisted technologies}

During the preparation of this manuscript, the author used OpenAI ChatGPT
as an AI-assisted research and writing tool. ChatGPT was used to assist
with checking mathematical arguments, identifying possible gaps and
counterexamples, developing and refining proof details, improving the
mathematical exposition and English language, assisting with literature
searches, and preparing and checking the \LaTeX{} source.

All mathematical statements, proofs, references, and conclusions were
subsequently reviewed and verified by the author. The author made all
final mathematical and editorial decisions and takes full responsibility
for the content of the manuscript.

\end{document}